\documentclass[11pt]{article}

\usepackage{amsmath, amsfonts, amssymb}
\usepackage{mathrsfs}
\usepackage{theorem}

\newtheorem{thm}{Theorem}[section]
\newtheorem{prop}[thm]{Proposition}
\newtheorem{lem}[thm]{Lemma}
\newtheorem{cor}[thm]{Corollary}

{\theorembodyfont{\rmfamily} \newtheorem{rem}[thm]{Remark}}
{\theorembodyfont{\rmfamily} }

\newcommand{\ra}{\rightarrow}
\newcommand{\dis}{\displaystyle}

\def\R{\mathbb R}

\def\N{\mathbb N}

\def\d{\text{\rm{d}}}
\def\E{\mathbb E}
\def\p{\mathbb P}

\def\la{\langle}
\def\raa{\rangle}
\def\La{\Lambda}
\def\veps{\varepsilon}

\def\S{\mathcal S}

\newcommand{\we}{\wedge}
\newcommand{\wt}{\widetilde}

\newcommand{\fin}{\hspace*{\fill}\rule{0.3em}{1ex}}
\newenvironment{proof}{{\bf \noindent Proof.}}{\fin}

\numberwithin{equation}{section}

\begin{document}

\title{Strong solutions and strong Feller properties for regime-switching diffusion processes in an infinite  state space\footnote{Supported in
 part by NSFC (No.11301030), NNSFC(No.11431014), 985-project and Beijing Higher Education Young Elite Teacher Project.}}

\author{Jinghai Shao\footnote{School of Mathematical Sciences, Beijing Normal University, Beijing, China. Email: shaojh@bnu.edu.cn} }
\maketitle
\begin{abstract}
  We establish the existence and pathwise uniqueness of regime-switching diffusion processes in an infinite state space, which could be time-inhomogeneous and state-dependent. Then the strong Feller properties of these processes are investigated by using the theory of parabolic differential equations and dimensional-free Harnack inequalities.
\end{abstract}
\section{Introduction}
The regime-switching diffusion processes (RSDPs in abbreviation) can be viewed as diffusion processes in random environments which are characterized by continuous time Markov chains. The feature of such processes is the inclusion of both continuous dynamics and discrete events at the same time. This, on the one hand, provides  more realistic models for many applications such as mathematical finance, wireless communication, biology and etc., (cf. \cite{YZ} and references therein); on the other hand, makes various properties of this system much more complicated. Refer to \cite{CH, PP, PS, SX1, Sh-a, Sh-b, YZ} for recurrent properties of this system, and to \cite{Ghosh, MY, SX2, SYZ, XY, XZ, ZY} for stability and optimal control of this system.

Precisely, a regime-switching diffusion process owns two components $(X_t,\La_t)$, where the first component $(X_t)$ satisfies a stochastic differential equation (SDE) in $\R^d$, $d\geq 1$:
\begin{equation}\label{1.1}
  \d X_t=b(t,X_t,\La_t)\d t+\sigma(t,X_t,\La_t)\d W_t,
\end{equation}
where $(W_t)$ is a Brownian motion in $\R^d$, $\sigma$ is a $d\times d$-matrix, and $b$ is a vector in $\R^d$.
While for each fixed $x\in \R^d$, the second component $(\La_t)$ is a jumping process on the state space $\S=\{1,2,\ldots,N\}$, $2\leq N\leq \infty$, satisfying
\begin{equation}\label{1.2}
\p(\La_{t+\delta}=l|\La_t=k, X_t=x)=\left\{\begin{array}{ll} q_{kl}(x)\delta+o(\delta), &\text{if}\ k\neq l,\\
                                       1+q_{kk}(x)\delta+o(\delta), & \text{if}\ k=l,
                         \end{array}\right.
\end{equation}
for $\delta>0$.
Throughout this work, for each $x\in \R^d$, $Q$-matrix $Q_x=(q_{kl}(x))$ is assumed to be irreducible and conservative (i.e. $q_i(x)=-q_{ii}(x)=\sum_{j\neq i}q_{ij}(x), \ \forall \,i\in\S$). When $(q_{ij}(x))$ is independent of $x$ and $(\La_t)$ is independent of $(W_t)$, $(X_t,\La_t)$ is said to be a state-independent regime-switching process; otherwise, it is called a state-dependent one.

Regime-switching diffusion processes $(X_t,\La_t)$ in  a finite state space (i.e. $\S$ is a finite set) have been relatively well studied. For example, in \cite{Sk}, Skorokhod has studied the asymptotic properties of this system. Recently, the books \cite{MY} and \cite{YZ} provide good surveys on the study of regime-switching processes. In particular, \cite{MY} focuses on the state-independent regime-switching diffusion process, but \cite{YZ} is mainly interested in state-dependent one.
However, limited work has been done in the study of RSDP in an infinite state space.
The existence of weak solution of RSDP in an infinite state space can be established similar to \cite{Sk}. The transience, recurrence, exponential ergodicity and the stability of RSDP in an infinite state space have been studied in
\cite{Sh-a, Sh-b, SX2}. These works reveal the essential difference between the study of RSDP in a finite state space and  that of RSDP in an infinite state space.

The motivation of this work is to study the strong Feller property of RSDPs in an infinite state space. But, to this aim, only the existence of weak solution is not enough.  Therefore, we first investigate the existence and uniqueness of strong solution of RSDPs in Section 2, which is of great meaning by itself.  Then, in Section 3, we study the strong Feller property of RSDPs in an infinite state space which extends the results in \cite{ZY}. Moreover, a result on the relationship between strong Feller property of $(X_t,\La_t)$ and strong Feller property of $(X_t^{(i)})$, $i\in\S$, is established for state-independent RSDPs.   Here, $(X_t^{(i)})$, defined by (\ref{env-x}), denotes the corresponding diffusion process of $(X_t)$ in the fixed environment $i$. We show that under some suitable conditions, if every $(X_t^{(i)})$, $i\in\S$, owns strong Feller property, then so does $(X_t,\La_t)$.  Conversely, if there exists some $i\in\S$ such that all $(X_t^{(j)})$ with $q_{ij}>0$ have strong Feller property but $(X_t^{(i)})$ does not have such property, then $(X_t,\La_t)$ does not have strong Feller property either.

The RSDPs have been used to investigate the weakly coupled elliptic systems. For instance,  in \cite{CZ1}, Chen and Zhao used the Dirichlet form theory to establish the existence of regime-switching diffusion processes in a finite state space. Then they used it to investigate the following weakly coupled elliptic system:
for $u=(u_1,\ldots,u_N):\R^d\ra \R^N$ with $N<\infty$, the weakly coupled elliptic operator:
\begin{equation*}
  \mathscr Au:=\begin{pmatrix}
    L^{(1)}&\ &\ &\ \\
    \ & L^{(2)}&\ &\ \\
    \ &\ & \cdots&\ \\
    \ &\ &\ &L^{(N)}
  \end{pmatrix} u+Qu,
\end{equation*}
where for each $k=1,2,\ldots, N$, $L^{(k)}$ is a strictly elliptic operator, i.e. $\dis L^{(k)}=\frac 12\sum_{i,j=1}^d\frac{\partial }{\partial x_i}\big(a_{ij}^{(k)}\frac{\partial}{\partial x_j}\big)+\sum_{i=1}^db_i^{(k)}\frac{\partial }{\partial x_i}$, and $Q$ is an $N\times N$ matrix-valued function. Indeed, by Skorokhod \cite{Sk}, the operator $\mathscr A$ coincides with the infinitesimal generator of $(X_t,\La_t)$ with $\S=\{1,2,\ldots,N\}$ and $N<\infty$.  They proved the solvability of the Cauchy problem of the system
\[\frac{\partial u}{\partial t}=\mathscr Au,\]
and gave a probabilistic representation theorem for solutions of the Dirichlet boundary value problem of $\mathscr A u=0$. Moreover, in \cite{CZ2}, potential theory for this elliptic system was studied, and some conditions in \cite{CZ1} were weakened. It is easy to see that our existence result on $(X_t,\La_t)$ in Section 2 can provide a probabilistic representation for above elliptic systems in the case $N=\infty$, which immediately helps us to get corresponding results in  \cite{CZ1} for elliptic systems $\mathscr A$ with $N=\infty$.

The paper is organized as follows. In Section 2, we establish the existence and uniqueness of strong solution of RSDP.  The basic technique is to represent the $Q$-process $(\La_t)$ in terms of a stochastic differential equation with respect to (w.r.t.) a Poisson random measure, which has been widely used (see, for instance,
\cite[Section II-2.1]{Sk}, \cite{Ghosh}, \cite{YZ}). Based on this representation, we  apply the tools of stochastic analysis to establish the existence and uniqueness of RSDP, which also extends the study of SDE with degenerate diffusion coefficients.

In  Section 3, we study the strong Feller property of $(X_t,\La_t)$. This property for RSDPs in a finite state space has been studied in many works such as \cite{Ghosh, XY, XZ, ZY}.  In this work, we generalize \cite{ZY}'s method to study the strong Feller property for time-homogeneous RSDPs in an infinite state space (see Theorem \ref{t-strong} below). Moreover, we present a general result on the relationship between strong Feller property of $(X_t,\La_t)$ and strong Feller property of corresponding diffusion processes in every fixed environment. See Theorem
\ref{t-s-2} below.  Based on this result, the known results on strong Feller property of diffusion processes under H\"omander's conditions can be easily extended to deal with state-independent RSDPs.

In Section 4, we establish the dimensional-free Harnack inequalities for state-independent regime-switching diffusion processes, and then apply them to study the strong Feller property of corresponding  processes.
This method can deal with time-inhomogeneous state-independent RSDPs, but can not deal with state-dependent RSDPs at present stage. Our method relies on the construction of coupling processes of RSDPs, but the construction of coupling processes for state-dependent RSDPs is rather difficult. We have made some research on this topic in \cite{XS}, but more work is needed to the aim of establishing Harnack inequalities.
Dimension-free Harnack inequality has been widely studied for SDEs and stochastic functional differential equations. Refer to \cite{Wang, SWY, WY} and references therein for more discussions on this inequality.

\section{Existence and uniqueness of strong solution}
In this section, we shall study the existence and uniqueness of strong solution of regime-switching diffusion processes in an infinite state space. First we introduce the representation of $(\La_t)$ in terms of the Poisson random measure similar to the one  introduced in
\cite[Chapter II-2.1]{Sk} or \cite{Ghosh, YZ} for $(\La_t)$ in a finite state space.

Precisely, for each $x\in \R^d$,
we construct a family of  intervals $\{\Gamma_{ij}(x);\ i,\,j\in \S\}$ on the real line in the following manner:
\begin{align*}
  \Gamma_{12}(x)&=[0,q_{12}(x)),\\
  \Gamma_{13}(x)&=[q_{12}(x),q_{12}(x)+q_{13}(x)),\\
  \cdots\\
  \Gamma_{21}(x)&=[q_{1}(x),q_{1}(x)+q_{21}(x)),\\
  \Gamma_{23}(x)&=[q_{1}(x)+q_{21}(x),q_{1}(x)+q_{21}(x)+q_{23}(x)),\\
  \cdots\\
  \Gamma_{31}(x)&=[q_1(x)+q_2(x),q_1(x)+q_2(x)+q_{31}(x)),
\end{align*}
and so on. For convenience of notation, we set $\Gamma_{ii}(x)=\emptyset$ and $\Gamma_{ij}(x)=\emptyset$ if $q_{ij}(x)=0$ for $i\neq j$. For each fixed $x$, these $\big\{\Gamma_{ij}(x)\big\}_{ij}$ are disjoint intervals, and the length of $\Gamma_{ij}(x)$ ($i\neq j$) equals to $q_{ij}(x)$. Define a function $h:\R^d\times \S\times \R\ra \R$ by
\begin{equation}\label{h-funct}
  h(x,i,z)=\left\{\begin{array}{ll} j-i&\quad \text{if}\ z\in \Gamma_{ij}(x),\\
                                   0& \quad \text{otherwise}.\end{array}\right.
\end{equation}
Then the process $(\La_t)$ can be expressed by the following SDE
\begin{equation}\label{1.3}
\d \La_t=\int_\R h(X_t,\La_{t-},z)N(\d t, \d z),
\end{equation}
where $N(\d t,\d z)$ is a Poisson random measure with intensity $\d t\times \d z$ on $[0,\infty)\times \R$, and independent of Brownian motion $(W_t)$ given by (\ref{1.1}). Set $\wt N(\d t,\d z)=N(\d t,\d z)-\d t\d z$.
In this work, in addition to the assumption that $Q$-matrix $(q_{ij}(x))$ is irreducible and conservative, we also need the following assumptions on the $Q$-matrix $(q_{ij}(x))$:
\begin{itemize}
  \item[$(\textrm{A1})$] There exists a positive constant $\kappa$ such that for every $i\in\S$ and every $x\in \R^d$, it holds $q_{ij}(x)=0$ for any $j\in\S$ with $|j-i|>\kappa$.
  \item[$(\textrm{A2})$] There exists a constant $c_q>0$ such that
  \[|q_{ij}(x)-q_{ij}(y)|\leq c_q|x-y|,\quad \forall \, x,y\in \R^d, \ i,j\in\S.\]
\end{itemize}
Now, we prepare a lemma on the Lipschitz continuity of the jump process which plays an important role in the existence and uniqueness of strong solution of SDE \eqref{1.1}, \eqref{1.2}.

\begin{lem}\label{key-lem}
  Suppose that (A1) and (A2) hold, then for $p>0$,
  \begin{equation}\label{ine-h}
    \int_{\R} |h(x,i,z)-h(y,i,z)|^p\d z\leq 2\kappa^{p+1}(\kappa+2i)c_q|x-y|,\quad \forall\, x,y\in \R^d,\ i\in \S.
  \end{equation}
\end{lem}

\begin{proof}
  In order to make the idea clear, we first consider a simple case, that is, $\S=\{1,2\}$. In this case by noting $\Gamma_{12}(x)$ and $\Gamma_{21}(x)$ are consecutive left-closed, right-open interval on $[0,\infty)$ with length $q_{12}(x)$ and $q_{21}(x)$ respectively, we obtain that, for $x\neq y\in\R^d$,
  \begin{align*}
    \int_{\R}|h(x,1,z)-h(y,1,z)|\d z&=\int_{\R}|\mathbf{1}_{\Gamma_{12}(x)}(z)-\mathbf{1}_{\Gamma_{12}(y)}(z)|\d z\\
    &=|q_{12}(x)-q_{12}(y)|\leq c_q|x-y|,
  \end{align*}
and
\begin{align*}
  \int_{\R}|h(x,2,z)-h(y,2,z)|\d z&=\int_{\R}(2-1)|\mathbf{1}_{\Gamma_{21}(x)}(z)-\mathbf{1}_{\Gamma_{21}(y)}(z)|\d z\\
  &=\int_{\R}\mathbf{1}_{(\Gamma_{12}(x)\Delta\Gamma_{12}(y))
         \cup(\Gamma_{21}(x)\Delta\Gamma_{21}(y))}(z)\d z\\
  &=|q_{12}(x)-q_{12}(y)|+|q_{12}(x)+q_{21}(x)-q_{12}(y)-q_{21}(y)|\\
  &\leq 2|q_{12}(x)-q_{12}(y)|+|q_{21}(x)-q_{21}(y)|\leq 3c_q|x-y|,
\end{align*}
where $A\Delta B=(A\backslash B)\cup(B\backslash A)$ for subsets $A$, $B$ of $\R$. Via studying this simple case, we show that the length of $\Gamma_{ij}(x)\Delta\Gamma_{ij}(y)$ can be estimated by $|x-y|$, but their coefficients are different due to the arrangement of $\Gamma_{ij}(x)$, $i,j=1,2$.

Next, we consider the general case $\S=\{1,2,\ldots,N\}$, $N\leq \infty$.
For $i\in \S$, $x\neq y$,
\begin{align*}
  &\int_{\R}|h(x,i,z)-h(y,i,z)|^p\d z
   =\int_{\R}|\sum_{j\neq i}(j-i)(\mathbf{1}_{\Gamma_{ij}(x)}(z)-\mathbf{1}_{\Gamma_{ij}(y)}(z))|^p\d z\\
  &\leq \sum_{j\neq i}|j-i|^p\int_{\R}|\mathbf{1}_{\Gamma_{ij}(x)}(z)-\mathbf{1}_{\Gamma_{ij}(y)}(z)|\d z\\
  &=\sum_{0<|j-i|\leq \kappa}\!\!|j-i|^p\Big(\big|\sum_{k=1}^{i-1}q_{k}(x)+\! \sum_{\substack{k=1\\k\neq i}}^{j-1}\! q_{ik}(x)-\sum_{k=1}^{i-1}q_k(y)-\! \sum_{\substack{k=1\\k\neq i}}^{j-1} \!q_{ik}(y)\big|\\
  &\qquad\qquad\qquad\qquad+\big|\sum_{k=1}^{i-1}q_k(x)+\!\!\sum_{\substack{k=1\\k\neq i}}^j \!q_{ik}(x)-\!\sum_{k=1}^{i-1}q_k(y)-\! \sum_{\substack{k=1\\k\neq i}}^{j} \!q_{ik}(y)\big|\Big)\\
  &\leq\!  \sum_{0<|j-i|\leq \kappa}\!\!\!\!|j-i|^p\Big(2\big|\sum_{k=1}^{i-1}\!q_k(x)\!+\! \sum_{\substack{k=1\\k\neq i}}^j \!q_{ik}(x)\!-\!\sum_{k=1}^{i-1}\!q_k(y)-\!\!\!\sum_{\substack{k=1\\k\neq i}}^{j}\! q_{ik}(y)\big|\!+\!|q_{ij}(x)\!-\!q_{ij}(y)|\Big)\\
  &\leq \kappa^p(\kappa-1)(2(2i+\kappa)-1)c_q|x-y|\\
  &\leq 2\kappa^{p+1}(\kappa+2i)c_q|x-y|.
\end{align*}
The proof is completed.
\end{proof}

Next, we give a priori estimate for $(X_t,\La_t)$.
\begin{prop}\label{m-est}
Let $(X_t,\La_t)$ be defined by (1.1), (1.3) with $(X_0,\La_0)=(x,i)$. Assume (A1) holds and further that
\begin{itemize}
  \item[$\mathrm{(A3)}$]$\la x, b(t,x,i)\raa\leq c(t)(1+|x|^2),\  \|\sigma(t,x,i)\|^2\leq c(t)(1+|x|^2)$, where $\|\sigma \|=\sqrt{\mathrm{trac}(\sigma\sigma^\ast)}$, $\sigma^\ast$ denotes the transpose of matrix $\sigma$, and $c(t)$ is a positive continuous function so that for each $T\in(0,\infty)$, $\int_0^Tc(t)\d t<\infty$;
  \item[$\mathrm{(A4)}$] There exist constants $\alpha,\,\beta\geq 0$ such that
  $q_{i}(x)\leq \alpha i+\beta |x|$, $\forall\, x\in \R^d,\, i\in \S$.
\end{itemize}
Then, for every $T\in(0,\infty)$,
\begin{align*}
&\E[\|X\|_{T}^2+\|\La\|_T^2]\\
&\leq \big(\frac{4}3|x|^2\!+\!4i^2\big) \exp\Big((4+\frac43 C_1^2)\!\int_0^T\!\!c(s)\d s+8\kappa^2(\alpha^2+\beta^2+2)(T+1)T\Big),
\end{align*}
where $\|X\|_t=\sup_{s\leq t} |X_s|$, $\|\La\|_t=\sup_{s\leq t}\La_s$, $t>0$, and $C_1$ is a positive constant determined by Burkholder-Davis-Gundy inequality.
\end{prop}

\begin{proof}
  Let $\tau_K=\inf\{t\geq 0;|X_t|+\La_t>K\}$, $K>0$. We always choose $K$ large enough so that $|X_0|+\La_0<K$. By It\^o's formula,
  \begin{equation}\label{1-0}
  \begin{split}
    |X_{t\wedge\tau_K}|^2&=|x|^2+2\int_0^{t\wedge \tau_K}\la X_s,b(s,X_s,\La_s)\raa \d s
                            +\int_0^{t\wedge \tau_K}\|\sigma(s,X_s,\La_s)\|^2\d s\\
                            &\quad +2\int_0^{t\wedge\tau_K}\la X_s,\sigma(s,X_s,\La_s)\d W_s\raa,
  \end{split}
  \end{equation}
  and \begin{equation}\label{1-0-1}
  \begin{split}
    \La_{t\we \tau_K}&=i+\int_0^{t\we \tau_K}\int_{\R}h(X_s,\La_{s-},z)N(\d s,\d z)\\
    &\leq i+\int_0^{t\we \tau_K}\int_{\R}h(X_s,\La_{s-},z)\wt N(\d s,\d z)+\kappa\int_0^{t\we \tau_K} q_{\La_{s-}}(X_s)\d s.
  \end{split}
  \end{equation}
  Since for any $t\in [0,T]$,
  \[\E\int_0^{t\we \tau_K}|X_s|^2\|\sigma(s,X_s,\La_s)\|^2\d s\leq K^2\int_0^tc(s)(1+K^2)\d s<\infty,\]
  this yields that $\big\{\int_0^{t\we \tau_K}\la X_s,\sigma(s,X_s,\La_s)\d W_s\raa\big\}_{t\in [0,T]}$ is a martingale. Similarly, as
  \begin{align*}\E\!\int_0^{t\we \tau_K}\!\!\!\int_{\R}h(X_s,\La_{s-},z)^2\d s\d z&\leq \kappa^2\E\!\int_0^{t\we \tau_K}\!\!q_{\La_{s-}}(X_s)\d s
   \leq \kappa^2\E\int_0^{t\we \tau_K}\!\!\alpha\La_{s-}\!+\!\beta|X_s|\d s <\infty,
  \end{align*}
  one gets that $\big\{\int_0^{t\we \tau_K} h(X_s,\La_{s-}, z)\wt N(\d s,\d z)\big\}_{t\in [0,T]}$ is a martingale.
  According to the Burkholder-Davis-Gundy inequality,
  \begin{equation}\label{1-1}
  \begin{split}
     \E\Big[\sup_{t\leq T}\big|\int_0^{t\we \tau_K}\la X_s,\sigma(s,X_s,\La_s)\d W_s\raa\big|\Big]
    &\leq C_1\E\sqrt{\int_0^{T\we\tau_K}|X_s|^2\|\sigma(s,X_s,\La_s)\|^2\d s}\\
    &\leq C_1\E\big[\|X\|_{T\we \tau_K}\sqrt{\int_0^{T\we \tau_K}\|\sigma(s,X_s,\La_s)\|^2\d s}\big]\\
    &\leq \frac 14\E\|X\|_{T\we \tau_K}^2+C_1^2\E\int_0^{T\we\tau_K}\|\sigma(s,X_s,\La_s)\|^2\d s\\
    &\leq \frac 14\E\|X\|_{T\we \tau_K}^2+C_1^2\E\int_0^{T\we \tau_K}c(s)(1+\|X\|_s^2)\d s.
    \end{split}
  \end{equation}
  Applying Burkholder-Davis-Gundy inequality again, we obtain
  \begin{equation}\label{1-2}
  \begin{split}
    &\E\Big[\sup_{t\leq T}\Big|\int_{0}^{t\we \tau_K}\int_{\R}h(X_s,\La_{s-},z)\wt N(\d s,\d z)\Big|^2\Big]\leq 4\E\int_0^{T\we \tau_K}\int_{\R}h^2(X_s,\La_{s-},z)\d s\d z\\
    &\leq 4\kappa^2 \E\int_0^{T\we \tau_K}q_{\La_{s-}}(X_s)\d s\leq 4\kappa^2\E\int_0^{T\we \tau_K}\big(\alpha\|\La\|_{s}+\beta \|X \|_s\big)\d s.
  \end{split}
  \end{equation}
  So, by (A3), (\ref{1-0-1}) and (\ref{1-2}), we obtain
  \begin{equation}
    \label{1-3}
    \begin{split}
      &\E\big[\|\La\|_{T\we \tau_K}^2\big]\\
      &\leq 3 i^2\!+\!3\E\Big[\Big(\sup_{t\leq T}\!\int_0^{t\we\tau_K}\!\!\!\int_{\R}\!\!h(X_s,\La_{s-},z)\wt N(\d s, \d z)\Big)^2\Big]\!+\!3\E\Big[\Big(\!\int_0^{T\we \tau_K}\!\!\!\int_{\R}\!\! h(X_s,\La_{s-},z)\d s\d z\Big)^2\Big]\\
      &\leq 3i^2\!+12\kappa^2\E\int_0^{T\we\tau_K}\!\!\!\! (\alpha\|\La\|_s\!+\!\beta\|X\|_s)\d s+3\kappa^2\E\Big[\Big(\int_0^{T\we \tau_K}\!\!\!(\alpha\|\La\|_s\!+\!\beta\|X\|_s)\d s\Big)^2\Big]\\
      &\leq 3i^2+6\kappa^2\E\int_0^{T\we\tau_K}((2\alpha+\beta)\|\La\|_s^2+\beta\|X\|_s^2)\d s+6\kappa^2T\E\int_0^{T\we\tau_K}(\alpha^2\|\La\|_s^2+\beta^2\|X\|_s^2)\d s\\
      &\leq 3i^2+6\kappa^2(T+1)(\alpha^2+\beta^2+2)
      \E\int_0^{T\we\tau_K}\!\!(\|\La\|_s^2+\|X\|_s^2)\d s,
    \end{split}
  \end{equation}
  where in the third inequality we have used $ 2\|X\|_s\leq  \|\La\|_s^2+\|X\|_s^2 $ as $\|\La\|_s\geq 1$.

  Consequently, combining (\ref{1-0}), (\ref{1-1}) with (\ref{1-3}), we get
  \begin{align*}
    &\E\big[\|X\|_{T\we\tau_K}^2+\|\La\|_{T\we\tau_K}^2\big]\\
    &\leq |x|^2+3\E\int_0^{T\we \tau_K}c(s)(\|\La\|_s^2+\|X\|_s^2)\d s+\frac 14\E\|X\|_{T\we \tau_K}^2\\
    &\quad +C_1^2\E\int_0^{T\we\tau_K}c(s)(\|\La\|_s^2+\|X\|_s^2)\d s+3i^2\\
    &\quad+6\kappa^2(T+1)(\alpha^2+\beta^2+2)
      \E\int_0^{T\we\tau_K}\!(\|\La\|_s^2+\|X\|_s^2)\d s
  \end{align*}
  Then, using  Gronwall's inequality, we get
  \[\E[\|X\|_{T\we\tau_K}^2\!+\|\La\|_{T\we\tau_K}^2]\leq \big(\frac{4}3|x|^2\!+\!4i^2\big) \exp\Big((4+\frac43 C_1^2)\!\int_0^T\!\!c(s)\d s+8\kappa^2(\alpha^2+\beta^2+2)(T+1)T\Big).\]
  Letting $K\ra \infty$, we obtain the desired result.
\end{proof}

Next, we consider the existence and uniqueness of strong solution of SDE \eqref{1.1}, \eqref{1.2} with non-Lipschitz coefficients.
To this aim, we introduce a class of functions:
\begin{equation}\label{u-class}
  \mathscr U:=\big\{u\in C^1((0,\infty);[1,\infty));\ \int_0^1\frac{\d s}{s u(s)}=\infty,\,\liminf_{r\downarrow 0}(u(r)+ru'(r))>0\big\}
\end{equation}
Here, the restriction that $u\geq 1$ is technical, otherwise we can replace it with $\max\{u,1\}$.
Refer to \cite{FZ} \cite{SWY} for existence and uniqueness of strong solutions of SDEs and stochastic functional differential equations under this type of non-Lipschitz coefficients.

\begin{thm}\label{t-unique}
Assume that (A1), (A2), (A3) hold, and for some constant $\alpha>0$,
\begin{equation}\label{con-1-2}
\sup_{x\in\R^d}q_i(x)\leq \alpha i,\quad \forall  \, i\in \S.
\end{equation}
Suppose
\begin{itemize}
  \item[$(\mathrm{A5})$] there exist  $u\in \mathcal U$ and increasing functions $C_i(t)\in C([0,\infty);(0,\infty))$, $i\in \S$, satisfying $\int_0^T C_i(t)\d t<\infty$ for all $T>0$, such that for all $t\geq 0$, $x,\,y\in \R^d$, $i\in\S$,
      \[\la x-y, b(t,x,i)-b(t,y,i)\raa+\frac 12\|\sigma(t,x,i)-\sigma(t,y,i)\|^2\leq C_i(t)|x-y|^2u(|x-y|^2).\]
\end{itemize}
Then there exists a unique strong solution of SDE (\ref{1.1}) and (\ref{1.2}) with $(X_0,\La_0)=(x,i_0)\in \R^d\times \S$.
\end{thm}

Before proving this theorem, we prepare two useful lemmas, which extend the corresponding results in \cite{XZ} to RSDPs in an infinite state space. Define a family of  auxiliary processes $(\xi_t^K)$ for $K=1,2,\ldots$, which are  time-homogeneous Markov chains on $\S$ such that
\begin{equation}\label{p-xi}
\p(\xi^K_{t+\delta}=j|\xi^K_t=i)=\left\{\begin{array}{ll}\alpha K\delta+o(\delta),&\ \text{if}\  0<|j-i|\leq \kappa,\ j\geq 1,\\
1-(\kappa\we(i\!-\!1)\!+\!\kappa) \alpha K\delta+o(\delta),&\ \text{if}\ j=i,\end{array}\right.
\end{equation} for $\delta>0$ small enough.
Denote by $\{p_{\xi,K}(t,i,j);\ t\geq 0,\,i,j\in\S\}$ the transition function of the Markov chain $(\xi_t^K)$.
By the theory of jump process (cf. \cite[Section 1.2]{Chen}), it holds
\begin{equation}\label{xi-est}
\lim_{t\ra 0}\frac{\log p_{\xi,K}(t,i,i)}{t}=-(\kappa\we (i\!-\!1)+\kappa)\alpha K,\quad i\in \S.
\end{equation}

\begin{lem}\label{lem-2}
Let $(X_t,\La_t)$ satisfy (\ref{1.1}), (\ref{1.2}) with initial condition $(X_0,\La_0)=(x,i_0)$.
Assume (A1), (A2) and (\ref{con-1-2}) hold. Then for each $K=1,2,\ldots$, and  every $t>0$,
\begin{equation}\label{est-1}
\begin{split}
&\p(\La_{t+2\delta}=k,\La_{t+\delta}=k|\La_t=k,X_t=x)\geq \p(\xi^K_{\delta}=k|\xi^K_{0}=k)^2,
\quad \forall\,1\leq k\leq K,
\end{split}
\end{equation}
for $\delta>0$ small enough.
\end{lem}

\begin{proof}
  By the definition of $(\xi_t)$ and (\ref{1.2}), we have for any $s>0$,
  \[\p(\La_{s+\delta}=k|\La_s=k,X_s=x)\geq \p(\xi^K_{s+\delta}=k| \xi^K_s=k),\quad \forall\,1\leq k\leq K, \ \forall\,x\in \R^d\] for $\delta>0$ small enough.
  Then, by the Markov property of $(X_t,\La_t)$, we obtain
  \begin{align*}
    &\p(\La_{t+2\delta}=k,\La_{t+\delta}=k|\La_t=k,X_t=x)\\
    &= \p(\La_{t+2\delta}=k\big|\La_{t+\delta}=k, X_{t+\delta}\in \R^d)\p(\La_{t+\delta}=k|\La_t=k, X_t=x)\\
    &\geq\p(\xi_{t+2\delta}^K=k\big|\xi_{t+\delta}^K=k)\p(\xi_{t+\delta}^K=k|\xi_t^K=k)\\
    &=\p(\xi_{\delta}^K=k|\xi_0^K=k)^2, \quad \forall\,1\leq k\leq K.
  \end{align*}
  where the time-homogeneous property of $(\xi_t^K)$ is used in the last step.
\end{proof}

\begin{lem}\label{lem-3}
Under the same assumptions as that of Lemma \ref{lem-2}, for each $K=1,2,\ldots$, it holds
\begin{equation}\label{est-2}
\p(\eta\geq t|\La_0=k,X_0=x)\geq \exp\big(-\!(\kappa\we\!(k\!-\!1)\!+\!\kappa)\alpha K t\big),\quad t>0,\ 1\leq k\leq K, x\in \R^d,
\end{equation}
where $\eta=\inf\{t>0;\ \La_t\neq \La_0\}$. This yields further that \[\lim_{t\ra 0}\p(\eta\geq t|\La_0=k,X_0=x)=1
\] uniformly for $1\leq k\leq K$ and $x\in \R^d$.
\end{lem}

\begin{proof}
  Due to the right continuity of the paths of $(X_t,\La_t)$, applying Lemma \ref{lem-2}, we obtain that
  \begin{align*}
    &\p(\eta\geq t|\La_0=k,X_0=x)
    =\p(\La_u=k,0\leq u\leq t|\La_0=k,X_0=x)\\
    &=\lim_{n\ra \infty}\p\big(\La_{\frac{mt}{2^n}}=k,1\leq m\leq 2^n\big|\La_0=k,X_0=x\big)\\
    &\geq\lim_{n\ra \infty}\p(\xi^K_{t/2^n}=k|\xi_0^K=k)^{2^n}
    =\lim_{n\ra \infty}\exp\Big(\frac{t\log p_{\xi,K}( t/2^n ,k,k)}{t/2^n}\Big)\\
    &\geq \exp\big(-\!(\kappa\we\!(k\!-\!1)\!+\!\kappa)\alpha K t\big),
  \end{align*}
  which is the desired result.
\end{proof}

\noindent\textbf{Proof of Theorem \ref{t-unique}:}
  Using Lemma \ref{key-lem}, it is easy to see that conditions (A1-A2) together with the continuity of $x\mapsto b(x,i)$ and $ x\mapsto \sigma(x,i)$ for each $i\in \S$ ensure the existence of a weak solution of SDE (\ref{1.1}) and (\ref{1.2}) (cf. \cite[Theorem 175]{Situ}). Therefore, according to the Yamada-Watanabe principle (cf. \cite[Theorem 137]{Situ}), we only need to show that the pathwise uniqueness holds for SDE (\ref{1.1}), (\ref{1.2}) to get the desired strong solution.

  Let $(X_t,\La_t)$ and $(Y_t,\La_t')$ both be solutions of SDE (\ref{1.1}) (\ref{1.2}) with the same initial condition $X_0=Y_0=x$, $\La_0=\La_0'=i_0$. Set $Z_t=X_t-Y_t$ for simplicity. Let $\tau_K=\inf\{t\geq 0; |X_t|+|Y_t|+\La_t+\La_t'>K\}$, $K>0$, and
  \begin{equation}\label{meeting}
  \zeta=\inf\{t>0;\La_t\neq \La_t'\}.
  \end{equation}
In the following, we take $K$ large enough that $|x_0|+i_0<K/2$.
By (A5), if $\La_t=\La_t'$ for $t\leq T$, It\^o's formula yields that
\begin{equation}\label{est-4}
\begin{split}
\d |Z_t|^2&=2\la Z_t, b(t,X_t,\La_t)-b(t,Y_t,\La_t')\raa\d t+\|\sigma(t,X_t,\La_t)-\sigma(t,Y_t,\La_t')\|^2 \d t\\
&\quad +2\la Z_t,(\sigma(t,X_t,\La_t)-\sigma(t,Y_t,\La_t'))\d B_t\raa\\
&\leq 2 C_{\La_t}(t)|Z_t|^2u(|Z_t|^2)\d t+2\la Z_t,(\sigma(t,X_t,\La_t)-\sigma(t,Y_t,\La_t'))\d B_t\raa.
\end{split}
\end{equation}
On the other hand, $u\in \mathscr U$ yields that there are positive constants $\lambda$, $\rho_0$ such that
\[u(r)+ru'(r)\geq \lambda,\quad r\in [0,\rho_0].\]
Let
\[\Psi_\veps(r)=\exp\Big(\lambda \int_1^r \frac{\d s}{\veps+su(s)}\Big),\ \ r,\,\veps\geq 0.\]
Then for any $\veps>0$, we have $\Psi_\veps\in C^2([0,\infty))$ and
\begin{gather*}
  ru(r)\Psi_\veps'(r)=\frac{\lambda ru(r)}{\veps+ru(r)}\Psi_\veps(r)\leq \lambda \Psi_\veps(r),\\
  \Psi_\veps''(r)=\frac{\lambda^2-\lambda(u(r)+ru'(r))}{(\veps+ru(r))^2}\leq 0, \quad r\in [0,\rho_0].
\end{gather*}
By \eqref{est-4} and the It\^o's formula, we get
\begin{align*}
  &\d \Psi_\veps(|Z_t|^2)\leq \lambda C_{\Lambda_t}(t)\Psi_{\veps}(|Z_t|^2)\d t+2\Psi_\veps'(|Z_t|^2)\la Z_t,(\sigma(t,X_t,\La_t)-\sigma(t,Y_t,\La'_t))\d W_t\raa.
\end{align*}
Hence, by Gronwall's inequality, we obtain
\[\E\Psi_\veps(|Z_{t\wedge\tau_K\wedge \zeta}|^2)\leq e^{\lambda \bar Ct}\Psi_\veps(0),\] for $t\leq T$ and some constant $\bar C$ depending on $T$ and $K$. Letting $\veps\downarrow 0$ and noting $\Psi_0(0)=0$, we get
\[\E|X_{t\we\tau_K\we \zeta}-Y_{t\we \tau_K\we \zeta}|^2=\E|Z_{t\we\tau_K\we \zeta}|^2=0, \ t\in [0,T].\]
Hence, for each $t\in [0,T]$, $X_{t\we\tau_K\we \zeta}=Y_{t\we\tau_K\we \zeta}$ almost surely. The continuity of the paths of $(X_t)$ and $(Y_t)$ yields further that almost surely
\begin{equation}\label{eq-xy}
X_{t\we \tau_K\we \zeta}=Y_{t\we\tau_K\we \zeta}, \ \forall\, t\in[0,T].
\end{equation}
This means that before the separation time of $(\La_t)$ and $(\La_t')$, the processes $(X_t)$ and $(Y_t)$ have to move together before exiting the closed ball $\{z\in\R^d;\ |z|\leq K\}$.

Now we study the behavior of $(\La_t)$ and $(\La_t')$. Since
\[\La_{t\we \tau_K\we \zeta}-\La'_{t\we\tau_K\we \zeta}=\int_0^{t\we\tau_K\we\zeta}\!\!\!\int_{\R}\!(h(X_s,\La_{s-},z)-h(Y_s,\La'_{s-},z))N(\d s,\d z),\]
by (\ref{eq-xy}) and the definition of $\zeta$,  the integral of the right hand side of the previous equation equals to 0.  Invoking the right continuity of the paths of $(\La_t)$ and $(\La_t')$, we get  almost surely
\begin{equation}\label{eq-lam}
\La_{t\we \tau_K\we \zeta}=\La'_{t\we \tau_K\we \zeta},\quad \forall\,t\in [0,T].
\end{equation}
By Proposition \ref{m-est}, conditions (A3) and (\ref{con-1-2}) yield that $\lim_{K\ra \infty} \tau_K=\infty$ almost surely.
Thanks to (\ref{eq-xy}) and (\ref{eq-lam}), to complete the proof of this theorem, we only need to show that $\zeta=\infty$ almost surely, which is equivalent to show that for any constant $M>0$, $\zeta\we M=M$. Take a constant $M$ and introduce  $\gamma=\zeta\we M$, $B=\{\omega;\gamma(\omega)<M\}$. We claim that $\p(B)=0$.

Indeed, if $\p(B)>0$, then by (\ref{eq-xy}) and (\ref{eq-lam}), and taking $T, t, K$ large enough such that $t>M$, $\tau_K>M$ a.s., we have for almost surely $\omega\in B$,
\[X_s(\omega)=Y_s(\omega),\ \forall\,s\leq \zeta(\omega)<M,\quad \La_{\zeta(\omega)}=\La'_{\zeta(\omega)}.\]
Let
\[\eta_{\La}=\inf\{s>\gamma;\La_{s}\neq \La_{\gamma}\},\ \ \eta_{\La'}
=\inf\{s>\gamma;\La'_{s}\neq \La'_{\gamma}\}. \]
It is easy to see that both $\eta_{\La}$ and $\eta_{\La'}$ are stopping time, and not smaller than $\gamma$.
By Lemma \ref{lem-3}, there exists $\delta_0>0$ such that
\[\inf_{1\leq k\leq K, x\in \R^d }\p\big(\eta_{\La}\geq \gamma+\delta_0|\La_\gamma=k,X_\gamma=x\big)\geq 1-\frac 14\p(B),\]
and
\[\inf_{1\leq k\leq K, x\in \R^d }
\p\big(\eta_{\La'}\geq \gamma+\delta_0|\La'_\gamma=k,Y_\gamma=x\big)\geq 1-\frac 14 \p(B).\]
Moreover,
\[\p(\eta_{\La}\!>\!\gamma\!+\!\delta_0)=\!\int_{\R^d\times\S}\!\!
\p\big(\eta_{\La}\!>\!\gamma\!+\!\delta_0|\La_{\gamma}=k,X_\gamma=x\big)
\p\big((X_\gamma,\La_\gamma)\in(\d x,\d k)\big)\geq 1\!-\!\frac 14\p(B).
\]
Similarly,
\[\p(\eta_{\La'}>\gamma+\delta_0)\geq 1-\frac 14\p(B).\]
Therefore, we get that
\begin{equation}\label{1-4}
  \p\big(\{\eta_{\La'}>\gamma+\delta_0\}\cap B\big)\geq \p(\eta_{\La'}>\gamma+\delta_0)-\p(B^c)
  \geq \frac 34\p(B)>0,
\end{equation}
and, further that
\begin{equation}\label{1-5}
\begin{split}
&\p \big(\{\eta_{\La}>\gamma+\delta_0\}\cap\{\eta_{\La'}>\gamma+\delta_0\}\cap B\big)\\
&\geq \p\big(\eta_{\La}>\gamma+\delta_0\big)-1+\frac 34\p(B)\geq \frac 12\p(B)>0.
\end{split}
\end{equation}
Let $\tilde\eta=\min\{\eta_{\La},\eta_{\La'}\}$. As $\La_{\gamma}=\La'_{\gamma}$, we know $\La_u=\La'_u$ for any $\zeta\leq u\leq \tilde\eta$.
Define a new stopping time $\tilde \zeta$ by
\[\tilde \zeta=\tilde \eta\,\mathbf{1}_{\zeta\leq M}+\zeta\mathbf{1}_{\zeta>M}.\]
Then (\ref{1-5}) implies that $\p\big(\{\tilde \zeta>\zeta\}\cap B\big)>0$, which means that
there exists a subset of $B$ of positive probability, such that
\[\zeta<\tilde \zeta,\quad \text{and}\ \forall \,t\leq \tilde \zeta,\ \La_t=\La_t'.\]
But this contradicts the definition of $\zeta$, which requires $|\La_t-\La_t'|>0$ for points close to $\zeta$ from the right. We complete the proof of this theorem.
\fin

\section{Strong Feller properties for time-homogeneous RSDPs}
In this section, we go to study strong Feller property of time-homogeneous regime-switching diffusion process $(X_t,\La_t)$ satisfying
\begin{equation}\label{1.1-h}
\d X_t=b(X_t,\La_t)\d t+\sigma(X_t,\La_t)\d W_t.
\end{equation}
In next section, we shall study strong Feller property of time-inhomogeneous, state-independent regime-switching diffusion process via dimension-free Harnack inequalities.
Corresponding to $(X_t,\La_t)$, there exists a family of diffusion processes $(X_t^{(i)})$, $i\in \S$, being the solutions of SDEs
\begin{equation}\label{env-x}
\d X_t^{(i)}=b(X_t^{(i)},i)\d t+\sigma(X_t^{(i)},i)\d W_t,\ \ i\in \S.
\end{equation}
$(X_t^{(i)})$ represents the behavior of $(X_t)$ in the fixed state $i$.
Various properties of $(X_t,\La_t)$ are closely related to the family of processes $\{(X_t^{(i)})_{t\geq 0}; i\in \S\}$.
Refer to, for example, \cite{CH, PS, SX1, Sh-a} for the study of recurrent property and stability of $(X_t,\La_t)$ in terms of the behavior of $(X_t^{(i)})$.
In the study of strong Feller property for $(X_t,\La_t)$ in a finite state space, \cite{XZ} established this property under the assumption that for each
$i\in \S$,  $(X_t^{(i)})$ has strong Feller property and the transition density exists; \cite{ZY} used the results on parabolic differential equations to establish this property. An important condition   posed in \cite{ZY} is a uniformly elliptic condition for each $(X_t^{(i)})$. This condition guarantees that $(X_t^{(i)})$ owns strong Feller property.

In the rest of this section, we first extend \cite{ZY}'s result to time-homogeneous RSDP in an infinite state space. Then for state-independent regime-switching diffusion process, we provide a general result on the relationship between the strong Feller property of $(X_t,\La_t)$ with the strong Feller property of $(X_t^{(i)})$, $i\in \S$.

\begin{thm}\label{t-strong} Let $(X_t,\La_t)$ be a time-homogeneous RSDP satisfying (\ref{1.1-h}) and (\ref{1.2}).
Assume that (A1), (A2), (A3) and (\ref{con-1-2}) hold. Suppose there exists a constant $c>0$ so that
\begin{align*}
  |b(x,i)-b(y,i)|\leq c|x-y|,\quad \|\sigma(x,i)-\sigma(y,i)\|\leq c|x-y|,\quad \forall\,x,y\in\R^d,\ i\in \S,
\end{align*} and  $a(x,i):=\sigma(x,i)\sigma(x,i)^\ast$ satisfies
\[\la a(x,i)\xi,\xi\raa\geq \lambda |\xi|^2,\quad \xi\in \R^d,\]
for some constant $\lambda>0$ and for all $(x,i)\in \R^d\times \S$. Then the process $(X_t,\La_t)$ has strong Feller property.
\end{thm}

\begin{proof}
  Similar to \cite{ZY}, we shall use the truncated method to prove this theorem. But, as $\S$ is an infinite state space, we also need to construct suitable truncated jump processes to derive the desired result now.  For $K=1,2,\ldots$, let $\tau_K=\inf\{t\geq 0; |X_t|+\La_t\geq K\}$.
  Define the $K$-truncated process $(X_K(t),\La_K(t))$ so that $(X_K(t), \La_K(t))=(X_t,\La_t)$ up to $\tau_K$.

  Let $\phi^K(x)$ be a smooth function with range $[0,1]$ satisfying $\phi^K(x)=1$ for $|x|\leq K$ and $\phi^K(x)=0$ for $|x|\geq K+1$. For $j,k=1,2,\ldots,d$, define
  \[a_{jk}^K(x,i)=a_{jk}(x,i)\phi^K(x),\quad b_j^K(x,i)=b_j(x,i)\phi^K(x).\]
  For $i\in\S$ and $i\leq K+\kappa$, \[q_{ij}^K(x):=q_{ij}(x)\phi^K(x) \ \text{for}\  j\leq K+\kappa;\ \   q_{i(K+\kappa+1)}^K(x):=\sum_{j\geq K+\kappa+1} q_{ij}(x)\phi^K(x).\]
  For $i=K\!+\!\kappa+1$,
  \begin{align*}&q_{ij}^K(x):=1+q_{ij}(x)\phi^K(x)\ \text{for}\ K+1\leq j\leq K+\kappa;\\
   &q_{(K+\kappa+1)(K+\kappa+1)}^K(x):=-\sum_{j=K+1}^{K+\kappa}\big(1+q_{ij}(x)\phi^K(x)\big).
   \end{align*}
  Since $(q_{ij}(x))$ is irreducible for each $x\in\R^d$, it is easy to check that $(q_{ij}^K(x))$ is also an irreducible $(K\!+\!\kappa\!+\!1)\!\times\!(K\!+\!\kappa\!+\!1)$-matrix for each $x\in \R^d$. $(q_{ij}^K(x))$ can be viewed as a $Q$-matrix on the space $\{1,2,\ldots,K\!+\!\kappa\!+\!1\}$, and coincides with $(q_{ij}(x))$ on $\{1,2,\ldots,K\}$.
  For any $g(\cdot,i)\in  C^2(\R^d)$, $i\in\S$, define the operator $\mathscr A^K$ by
  \begin{align*}
    \mathscr A^Kg(x,i)&=\frac 12\sum_{j,k=1}^da_{jk}^K(x,i)\frac{\partial^2}{\partial x_j\partial x_k} g(x,i)+\sum_{j=1}^db_j^K(x,i)\frac{\partial}{\partial x_j}g(x,i)\\
    &\quad +\sum_{j\neq i}q^K_{ij}(x)(g(x,j)-g(x,i)).
  \end{align*}
  Denote by $P_{x,i}^K$ the associated probability measure for $\mathscr A^K$, and $\E_{x,i}^K$ the corresponding expectation. By Theorem  \ref{t-unique}, the strong solution is unique, and hence for any bounded measurable function $f$ on $\R^d\times \S$,
  \[\E_{x,i}[f(X_t,\La_t)\mathbf 1_{\tau_K>t}]=\E_{x,i}^K[f(X_K(t),\La_K(t))\mathbf 1_{\tau_K>t}].\]
  By \cite[Theorem 3.10]{ZY}, $P_{x,i}^K$ has strong Feller property.
  By Proposition \ref{m-est}, we have
  \[K\p_{x,i}(\tau_K\leq t)\leq (\frac43|x|^2+4i^2)\exp\Big((4+\frac43C_1^2)\int_0^tc(s)\d s+8\kappa^2(\alpha^2+2)(t+1)t\Big), \]
  which implies that
  $\p_{x,i}(\tau_K\leq t)\ra 0$ as $K\ra \infty$  uniformly for $(x,i)$ in a compact set of $\R^d\times\S$.
  Therefore, for every fixed $(x,i)\in \R^d\times\S$,
  \begin{align*}
    &|\E_{y,i} f(X_t,\La_t)-\E_{x,i} f(X_t,\La_t)|\\
    &\leq |\E_{y,i} f(X_t,\La_t)-\E_{y,i}^Kf(X_K(t),\La_K(t))|+|\E_{y,i}^K f(X_K(t),\La_K(t))-\E_{x,i}^K f(X_K(t),\La_K(t))|\\
    &\quad +|\E_{x,i}^K f(X_K(t),\La_K(t))-\E_{x,i} f(X_t,\La_t)|\\
    &\leq \|f\|\p_{x,i}(\tau_K\leq t)+\|f\|\p_{y,i}(\tau_K\leq t)+|\E_{y,i}^K f(X_K(t),\La_K(t))-\E_{x,i}^Kf(X_K(t),\La_K(t))|,
  \end{align*}
  where $\|f\|$ denotes the essential supremum norm of $f$. Consequently, the desired strong Feller property follows immediately from the previous inequality.
\end{proof}

\begin{thm}\label{t-s-2}
Let $(X_t,\La_t)$ be a time-homogeneous, state-independent regime-switching diffusion process satisfying (\ref{1.1-h}), (\ref{1.2}). For each $i\in \S$, $(X_t^{(i)})$ is defined by (\ref{env-x}).
Assume that $(X_t,\La_t)$ and all $(X_t^{(i)})$, $i\in\S$, own Feller property.
\begin{itemize}
  \item[$1^\circ$] If for every $i\in\S$, $(X_t^{(i)})$ has strong Feller property, then $(X_t,\La_t)$ also has strong Feller property.
  \item[$2^\circ$] If there exists some $i\in \S$ such that for all $j\in\S$ with $q_{ij}>0$, $(X_t^{(j)})$ has strong Feller property, but  $(X_t^{(i)})$ \textbf{doesn't} have strong Feller property,
      then $(X_t,\La_t)$ \textbf{doesn't} have strong Feller property either.
\end{itemize}
\end{thm}

\begin{proof}
  We provide an explicit construction of probability space $(\Omega,\p)$ to make the role played by the state-independence of $(\La_t)$ clear.

  Let $\Omega_1=\{\omega:[0,\infty)\ra \R^d\ \text{continuous};\ \omega(0)=0\}$. Let $\p_1$ be the Wiener measure on $\Omega_1$. Then $w(t,\omega):=\omega(t)$ for $\omega\in\Omega_1$ is a standard Brownian motion $\p_1$.
  Let \[\Omega_2=\big\{\omega=\sum_{i=1}^n\delta_{t_i,u_i};\ n\in \mathbb{N}\cup\{\infty\}, \ (t_i,u_i)\in [0,\infty)\times[0,\infty)\big\}.\]
  There exists a probability measure $\p_2$ on $\Omega_2$ such that $N(\d t,\d u,\omega):=\omega(\d t,\d u)$ is a Poisson random measure with intensity $\d t\times \d u$.
  Let $\Omega=\Omega_1\times\Omega_2$, and $\p=\p_1\times\p_2$. Then under $\p$, for $\omega=(\omega_1,\omega_2)\in \Omega$, $(\omega_1(t))$ is a standard Brownian motion, $\omega_2(\d t,\d u)$ is a Poisson random measure with intensity $\d t\times \d u$. We can rewrite SDE (\ref{1.1-h}), (\ref{1.3}) in the following form:
  \begin{align*}
    \d X_t&=b(X_t,\La_t)\d t+\sigma(X_t,\La_t)\d \omega_1(t),\\
    \d \La_t&=\int_{\R}h(\La_{t-},u)\omega_2(\d t,\d u).
  \end{align*}
  Rewrite (\ref{env-x}) in the form: $\d X_t^{(i)}=b(X_t^{(i)},i)\d t+\sigma(X_t^{(i)},i)\d \omega_1(t)$.

  Let $\eta=\inf\{t>0;\La_t\neq \La_0\}$. By the theory of jump process, if $\La_0=i$, the distribution of $\eta$ is exponential distribution with parameter $q_i=\sum_{j\neq i}q_{ij}$. So $\p(\eta>0)=1$.

 $1^\circ$\  When $\La_0=i$, we know that $(X_t)$ coincides with $(X_t^{(i)})$ up to $\eta$. For $f\in \mathscr B_b(\R^d\!\times\! \S)$, by the strong Markov property, we obtain
\begin{equation}\label{3-3}
\begin{split}
  &P_t f(x,i)\\ 
  &=\E_{x,i}\big[\mathbf 1_{0<\eta<t}\E_{X_\eta,\La_\eta}\big[f(X_{t-\eta},\La_{t-\eta})\big]\big]
  +\E_{x,i}\big[\mathbf 1_{\eta\geq t}f(X_t,\La_t)\big]\\
  &=\E_{\p}\big[\mathbf 1_{0<\eta<t}\E_{X^{(i)}_\eta,\La_{\eta}}\big[f(X_{t-\eta},\La_{t-\eta})\big]\big]+\E_{\p}\big[\mathbf 1_{\eta\geq t} f(X_t,i)\big]\\
  &=\E_{\p_2}\big[\mathbf 1_{0<\eta<t}\E_{\p_1}\big[\E_{X_\eta^{(i)},\La_\eta}\big[f(X_{t-\eta},\La_{t-\eta})\big]\big]\big]
  +\E_{\p_2}\big[\mathbf 1_{\eta\geq t} \E_{\p_1}f(X_t^{(i)},i)\big]\\
  &=\E_{\p_2}\big[\mathbf 1_{0<\eta<t}P_{\eta}^{(i)}g_{\eta}^{(\La_\eta)}(x)\big]+\E_{\p_2}\big[\mathbf 1_{\eta\geq t}P_t^{(i)}f(\cdot,i)(x)\big],
\end{split}
\end{equation}
where $g_s^{(k)}(x):=\E_{x,k}[f(X_{t-s},\La_{t-s})]$ for $0<s<t$, $x\in \R^d$, $k\in \S$.
Here, $P_t^{(i)}$ denotes the semigroup corresponding to the process $(X_t^{(i)})$, i.e.
$P_t^{(i)} h(x)=\E_x[h(X_t^{(i)})]$ for bounded measurable function $h$ on $\R^d$. And $P_t^{(i)}f(\cdot,i)(x)$ is equal to $\E_{x}[f(X_t^{(i)},i)]$ used in the previous equation.
Clearly, $x\mapsto g_s^{(k)}(x)$ is bounded measurable. Since for each $i\in \S$, $(X_t^{(i)})$ has strong Feller property, then $x\mapsto P_s^{(i)}g_s^{(k)}(x)$ and   $x\mapsto P_t^{(i)}f(\cdot,i)(x)$ are all bounded continuous for every $i,k\in\S$. By the dominated convergence theorem, we obtain that the function
\begin{equation*}
x\mapsto \E_{\p_2}\big[\mathbf 1_{0<\eta<t}P_{\eta}^{(i)}g_{\eta}^{(\La_\eta)}(x)\big]+\E_{\p_2}\big[\mathbf 1_{\eta\geq t}P_t^{(i)}f(\cdot,i)(x)\big]=P_tf(x,i) \end{equation*}
is continuous. Note that in previous argument the strong Feller property of $P_t^{(i)}$ ensures that $x\mapsto P_tf(x,i)$ is continuous for $f\in \mathscr B_b(\R^d\!\times\!\S)$. This property will be used in next step. Therefore, the strong Feller property of $(X_t,\La_t)$ follows from the arbitrariness of $f$, $t$, and $(x,i)$.

$2^\circ$ \  Suppose $X_0=x_1$ and $\La_0=i$, and set $\eta=\inf\{t>0;\La_t\neq \La_0\}$. Since $(X_t^{(i)})$ has no strong Feller property, there exist $t_1>0$, $\tilde f\in \mathscr B_b(\R^d)$ and $x_1\in\R^d$ so that $x\mapsto P_{t_1}^{(i)}\tilde f(x)$ is not continuous at $x_1$.
Define \[\bar f(x,k)=\tilde f(x)\ \text{ for any $x\in\R^d$, $k\in\S$}.\]
Then, by noting $\p_2(\eta=t_1)=0$, we get
\begin{align*}
  &P_{t_1}\bar f(x,i)\\
  &=\E_{\p}[\mathbf1_{\eta> t_1}\bar f(X_{t_1},\La_{t_1})]+\E_{\p}[\mathbf 1_{\eta< t_1}\E_{X_{\eta},\La_{\eta}}[\bar f(X_{t_1-\eta},\La_{t_1-\eta})]\\
  &=\E_{\p_2}[\mathbf 1_{\eta>t_1}\E_{\p_1}[\bar f(X_{t_1}^{(i)},i)]]+\E_{\p}\big[\mathbf 1_{\eta<t_1}\E_{X_\eta^{(i)},\La_{\eta}}[\bar f(X_{t_1-\eta},\La_{t_1-\eta})]\big]\\
  &=\E_{\p_2}[\mathbf 1_{\eta>t_1}P_{t_1}^{(i)}\tilde f(x)]+\E_{\p_2}[\mathbf 1_{\eta<t_1}P^{(i)}_{\eta}g_{\eta}^{(\La_\eta)}(x)]\\
  &=:I+I\!I,
\end{align*}
where $g_s^{(k)}(x)=\E_{x,k}[\bar f(X_{t_1-s},\La_{t_1-s})]$ for   $x\in \R^d$, $k\in \S$, and $0<s<t_1$. By the right continuity of the paths of $(\La_t)$, we know that $\La_\eta\in\Theta:=\{j\in\S;\ q_{ij}>0\}$. By the assumption,   $(X_t^{(j)})$ has strong Feller property for $j\in\Theta$. Using the deduction in the previous step, we have $x\mapsto  g_s^{(j)}(x)$ is continuous for each $j\in \Theta$. Hence part $I\!I$ is a continuous function for $x$ due to the Feller property of $(X_t^{(i)})$. However, the first part $I$ is not continuous at point $x_1$. In all, $x\mapsto P_{t_1}\bar f(x,i)$ is not continuous at point $x_1$, which shows that $(X_t,\La_t)$ doesn't have strong Feller property.
\end{proof}

\section{Dimension-free Harnack inequality}
In this section, we shall establish dimension-free Harnack inequalities for state-independent regime-switching diffusion processes by using the coupling method. Then, these inequalities are applied to study the strong Feller property of RSDPs which could be time-inhomogeneous.
Precisely, let $(X_t,\La_t)$ be defined by
\begin{equation}\label{4.1}
\d X_t=b(t,X_t,\La_t)\d t+\sigma(t,X_t,\La_t)\d W_t, \quad X_0=x\in\R^d,
\end{equation}
and $(\La_t)$ is a $Q$-process in $\S=\{1,2,\ldots,N\}$, $2\leq N\leq \infty$ satisfying
\begin{equation}\label{4.2}\p(\La_{t+\delta}=k|\La_t=i)=\left\{\begin{array}{ll} q_{ik}\delta+o(\delta), &i\neq k,\\
                       1+q_{ii}\delta+o(\delta), &i=k,\end{array}\right.
                       \end{equation}
for $\delta>0$ small enough. $(\La_t)$ is independent of $(W_t)$, and $(q_{ij})$ is conservative and irreducible. Associated with $(X_t,\La_t)$, there is a family of Markov operators:
\[P_tf(x,i):=\E_{x,i}[f(X_t,\La_t)],\quad t\geq 0,\ (x,i)\in\R^d\times\S,\ f\in\mathscr B_b(\R^d\times\S).\]

We collect the assumptions used below to establish Harnack inequality.
\begin{itemize}
  \item[$(\mathrm{H1})$]\ There exist $u,\,\tilde u\in\mathscr U$, defined by \eqref{u-class}, with $u'\leq 0$ and increasing functions $C_i(t)$, $\wt C_i(t)\in C([0,\infty);(0,\infty))$, $i\in\S$ such that for all $t\geq 0$ and $x,\,y\in \R^d$, $i\in \S$,
      \begin{align*}
        &\la x-y, b(t,x,i)-b(t,y,i)\raa+\frac 12\|\sigma(t,x,i)-\sigma(t,y,i)\|^2\leq C_i(t)|x-y|^2u(|x-y|^2)\\
        &\|\sigma(t,x,i)-\sigma(t,y,i)\|^2\leq \wt C_i(t)|x-y|^2\tilde u(|x-y|)^2.
      \end{align*}
  \item[$(\mathrm{H2})$]\ There exists a decreasing function $\lambda\in C([0,\infty);(0,\infty))$ such that $|\sigma(t,x,i)y|\geq \lambda(t)|y|$, $t\geq 0$, $x,\,y\in\R^d$, $i\in\S$.
\end{itemize}

\begin{thm}\label{Harnack}
Assume that (H1), (H2) hold and there exists a constant $\alpha>0$ such that $q_i\leq \alpha i$ for every $i\in\S$.
Suppose that
\begin{equation}\label{con-4-1}
\sup_{t\geq 0,i\in\S}\{|b(t,0,i)|+\|\sigma(t,0,i)\|\}<\infty,\end{equation}
and  for each $t>0$,
\begin{equation}\label{con-4-2}
    0<\inf_{i\in \S} C_i(t)\leq\sup_{i\in \S} C_i(t)<\infty.
\end{equation}
\begin{itemize}
  \item[$1^\circ$]\ For any initial point $(x,i)$, the SDE (\ref{4.1}), (\ref{4.2}) has a unique solution, and the solution is non-explosive.
  \item[$2^\circ$]\ If for some constant $\gamma>0$,
  \begin{equation}\label{con-phi}
  \varphi(s):=\int_0^su(r)\,\d r\leq \gamma su(s)^2,\quad s\geq 0,
  \end{equation}
  then for each $T>0$ and strictly positive function $f\in \mathscr B_b(\R^d\times \S)$,
  \begin{equation}\label{H-ineq}
  P_T\log f(y,i)\leq \log P_T f(x,i)+\E\Big[\frac{C_{\La_T}(T)\varphi(|x-y|^2)}
  {\lambda(T)\big(1-\exp(-2C_{\La_T}(T)T/\gamma)\big)}\Big].
  \end{equation}
\end{itemize}
\end{thm}

\begin{proof}
  $1^\circ$\ The existence and uniqueness of solution for \eqref{4.1} and \eqref{4.2} can be proved by using Theorem \ref{t-unique}. But here we would like to provide another proof by using the idea of \cite{XY-11} to stress the advantage caused by the state-independence of $(\La_t)$.  In the argument, we use the probability space $(\Omega,\p)=(\Omega_1\times\Omega_2,\p_1\times\p_2)$ constructed in Theorem \ref{t-s-2}. The existence of the Markov chain $(\La_t)$ is well known (cf. \cite{Chen}).
  By the path property of the jump process $(\La_t)$, there exists a finite number of single jumps during any finite interval $[0,T]$ for almost every $\omega=(\omega_1,\omega_2)\in \Omega$. Hence, for $\p_2$-almost every $\omega_2\in \Omega_2$, there exists a finite number $m=m(\omega_2)\in \N$ so that
  \[0=\tau_0<\tau_1<\ldots<\tau_m\leq T,\]
  where $\tau_i=\inf\{t>\tau_{i-1};\La_t\neq \La_{\tau_{i-1}}\}$, $i\geq 1$. Note that by the construction of $(\Omega,\p)$, $\tau_i$ depends only on $\omega_2\in \Omega_2$ for each $i\geq 1$.

  When $\La_0=i$, on the interval $[\tau_0,\tau_1)$, \eqref{4.1} is equivalent to the following SDE,
  \begin{equation}\label{e-4-1}
  \d X_t^{(i)}=b(t, X_t^{(i)},i)\d t+\sigma(t,X_t^{(i)},i)\d W_t,\quad X_0^{(i)}=x,
  \end{equation}
  According to the theory of SDE without switching
  (see, for instance, \cite[Theorem 2.1]{SWY}), \eqref{e-4-1}
  has a unique solution on $[\tau_0,\tau_1)$. So $(X_t,\La_t)=(X_t^{(i)}, i)$ for $t\in [\tau_0,\tau_1)$. Set
  \[(X_t,\La_t)=\begin{cases}
    (X_t^{(i)},i),\ &\tau_0\leq t<\tau_1,\\
    (X_{\tau_1}^{(i)},\La_{\tau_1}),\ &t=\tau_1.
  \end{cases}\]

  Next, on the interval $[\tau_1,\tau_2)$, by the same reason as above, we have $(X_t,\La_t)=(X_t^{(\La_{\tau_1})},\La_{\tau_1})$, where $\big(X_t^{(\La_{\tau_1})}\big)$ is the unique solution of the following SDE
  \begin{equation}\label{e-4-2}
  \d X_t^{(\La_{\tau_1})}=b(t,X_t^{(\La_{\tau_1})},\La_{\tau_1})\d t+\sigma(t,X_t^{(\La_{\tau_1})},\La_{\tau_1})\d W_t,\quad X_0=X_{\tau_1}.
  \end{equation}
  Set
   \[(X_t,\La_t)=\begin{cases}
     (X_t^{(\La_{\tau_1})}, \La_{\tau_1}),\ &\tau_1\leq t<\tau_2,\\
     (X_{\tau_2}^{\La_{\tau_1}},\La_{\tau_2}),\ &t=\tau_2.
   \end{cases}\]
   Repeating this procedure, we see that (\ref{4.1}) has a unique solution on $[0,T]$.

  Since $u$ is decreasing, the first inequality in (H1) with $y=0$ implies that for $|x|\geq 1$
  \begin{align*}&2\la b(t,x,i),x\raa+\|\sigma(t,x,i)\|^2\\
  &\leq 2\la b(t,0,i),x\raa+\|\sigma(t,0,i)\|^2+2\|\sigma(t,0,i)\|\|\sigma(t,x,i)\|+C_i(t)|x|^2u(1).
  \end{align*}
  The second inequality in (H1) with $y=0$ implies that for $|x|\geq 1$
  \begin{align*}
    \|\sigma(t,x,i)\|&\leq \|\sigma(t,0,i)\|+\sum_{k=1}^{[|x|]}\|\sigma(t,\frac{kx}{[|x|]},i)
    -\sigma(t,\frac{(k-1)x}{[|x|]},i)\|\\
    &\leq \|\sigma(t,0,i)\|+2|x|\sqrt{\wt C_i(t) \tilde u(1)},
  \end{align*} where $[|x|]$ denotes the integer part of $|x|$.
  Invoking condition (\ref{con-4-1}), Proposition \ref{m-est} yields that $(X_t,\La_t)$ is non-explosive.

  $2^\circ$\  Assume $(X_t,\La_t)$ starts from $(x,i)$. Let $(Y_t)$ be the unique solution of (\ref{4.1}) with $Y_0=y$.
  Then
  \begin{equation}
    P_T\log f(y,i)=\E_{\p}\log f(Y_T,\La_T)=\E_{\p_2}\big[\E_{\p_1}\big[\log f(Y_T,\La_T)\big]\big].
  \end{equation}
  By the independence of $(\La_t)$ w.r.t. Brownian motion $(W_t)$, for almost every $\omega_2$, we can apply the Harnack inequality (see \cite[Theorem 2.1]{SWY}) to the SDE
  \[\d X_t(\omega_1,\omega_2)=b(t,X_t(\omega_1,\omega_2),\La_t(\omega_2))\d t+\sigma(t,X_t(\omega_1,\omega_2),\La_t(\omega_2))\d \omega_1(t)\]
  to yield that
  \begin{equation}\label{e-4-3}
  \begin{split}
  &\E_{\p_1}\big[\log f(Y_T,\La_T)\big]\\
  &\leq \log\E_{\p_1}\big[f(X_T, \La_T)\big]+\frac{C_{\La_T}(T)
  \varphi(|x-y|^2)}{\lambda(T)\big(1-\exp\big(-2C_{\La_T}(T)T/\gamma\big)\big)}.
  \end{split}
  \end{equation}
  Taking expectation w.r.t. $\p_2$ and using Jensen's inequality, we get
  \begin{align*}
    &P_T\log f(y,i)\\
    &\leq \log\E_{\p}f(X_T,\La_T)+\E_{\p_2}\Big[\frac{C_{\La_T}(T)
  \varphi(|x-y|^2)}{\lambda(T)\big(1-\exp\big(-2C_{\La_T}(T)T/\gamma\big)\big)}\Big]\\
  &\leq \log P_Tf(x,i)+\E\Big[\frac{C_{\La_T}(T)
  \varphi(|x-y|^2)}{\lambda(T)\big(1-\exp\big(-2C_{\La_T}(T)T/\gamma\big)\big)}\Big].
  \end{align*}
  The proof of this theorem is complete.
\end{proof}

\begin{cor}\label{cor}
Under the conditions of Theorem \ref{Harnack}, the process $(X_t,\La_t)$ has strong Feller property.
\end{cor}

\begin{proof}
  This is a standard application of Harnack inequality (\ref{H-ineq}) (cf. \cite{Wang2}). For the convenience of the reader, we provide the proof. Let $f\in \mathscr B_b(\R^d\times\S)$ be  positive. Applying Harnack inequality (\ref{H-ineq}) to $1+\veps f$ in place of $f$ for $\veps>0$, we obtain
  \[P_T\log(1+\veps f)(y,i)\leq \log P_T(1+\veps f)(x,i)+\E\Big[\frac{C_{\La_T}(T)\varphi(|x-y|^2)}
  {\lambda(T)\big(1-\exp(-2C_{\La_T}(T)T/\gamma)\big)}\Big].\]
  By a Taylor expansion, this yields
  \begin{equation}\label{ine-4-1}
  \begin{split}
    &\log(1)+\veps P_Tf(y,i)+o(\veps)\\
    &\leq \veps P_Tf(x,i)+o(\veps)+\E\Big[\frac{C_{\La_T}(T)\varphi(|x-y|^2)}
  {\lambda(T)\big(1-\exp(-2C_{\La_T}(T)T/\gamma)\big)}\Big]
  \end{split}
  \end{equation}
  Letting $y\ra x$, we get
  \[\veps\limsup_{y\ra x}P_Tf(y,i)\leq \veps P_Tf(x,i)+o(\veps).\]
  Thus, \[P_Tf(x,i)\geq \limsup_{y\ra x}P_Tf(y,i)\quad \forall\, x\in \R^d,\, i\in\S.\]
  On the other hand, letting $x\ra y$ in (\ref{ine-4-1}), we get
  \[P_Tf(y,i)\leq \liminf_{x\ra y}P_Tf(x,i)\quad \forall\,y\in \R^d,\, i\in \S.\]
  Therefore, $x\mapsto P_Tf(x,i)$ is continuous, moreover, $(x,i)\mapsto P_Tf(x,i)$ is continuous due to the discrete topology of $\S$.
\end{proof}

\begin{rem}
  Compared with Theorem \ref{t-strong}, Corollary \ref{cor} can deal with the strong Feller property of time-inhomogeneous RSDPs. Due to the difficulty in the construction of successful coupling processes for state-dependent RSDPs, we can establish the Harnack inequality and further prove the strong Feller property for state-independent RSDPs at present stage. The establishing of Harnack inequalities for state-dependent RSDPs is also important and left open.
\end{rem}
\noindent\textbf{Acknowledgment:} The author is very grateful to the referee for his valuable comments which improves the quality of this work. And he would like to thank professor Jianhai, Bao for his useful discussion.

\end{document}